\documentclass[AMS,Garamond1COL]{WileyNJDv5} %

\articletype{Open Access}%
\received{Date Month Year}
\revised{Date Month Year}
\accepted{Date Month Year}
\journal{Journal of Graph Theory}
\volume{00} %
\copyyear{2023} %
\startpage{1}

\usepackage{graphicx} %
\usepackage{amsthm}
\usepackage{amsmath}
\usepackage{amssymb}

\usepackage[color=yellow]{todonotes}

\usepackage{xspace}
\usepackage{tikz-cd}
\usepackage{booktabs}

\usepackage{cleveref}
\usepackage{fontawesome5}

\usepackage{enumitem}

\newcommand{\td}{\ensuremath{\operatorname{td}}\xspace} %
\newcommand{\nd}{\ensuremath{\operatorname{nd}}\xspace} %
\newcommand{\cvdn}{\ensuremath{\operatorname{cvdn}}\xspace} %
\newcommand{\mw}{\ensuremath{\operatorname{mw}}\xspace} %
\newcommand{\mim}{\ensuremath{\operatorname{mim}}\xspace} %
\newcommand{\vccn}{\ensuremath{\operatorname{vccn}}\xspace} %
\newcommand{\eccn}{\ensuremath{\operatorname{eccn}}\xspace} %

\begin{document}

\title{Density of Traceable Graphs}

\author[1]{Michal Dvořák}
\author[1]{Dušan Knop}
\author[1]{Michal Opler}
\author[1]{Jan Pokorný}
\author[1]{Ondřej Suchý}
\author[1]{Krisztina Szilágyi}

\authormark{Michal Dvořák, Dušan Knop, Michal Opler, Jan Pokorný, Ondřej Suchý, Krisztina Szilágyi}
\titlemark{Density of Traceable Graphs}

\address[1]{\orgdiv{Faculty of Information Technology}, \orgname{Czech Technical University in Prague}, \orgaddress{\state{Czech Republic}, \country{Prague}}}

\corres{
\email{michal.dvorak@fit.cvut.cz}
\email{dusan.knop@fit.cvut.cz}
\email{michal.opler@fit.cvut.cz}
\email{jan.pokorny@fit.cvut.cz}
\email{ondrej.suchy@fit.cvut.cz}
\email{krisztina.szilagyi@fit.cvut.cz}}

\fundingInfo{%
European Union under the project Robotics and advanced industrial production (reg. no. CZ.02.01.01/00/22\_008/0004590).}

\abstract[Abstract]{
    In this paper, we consider the minimum number of edges of traceable graphs, i.e. graphs that have a Hamiltonian path, for graphs that have a specific structure. Clearly, if we do not impose any additional restrictions, the minimum number of edges of an $n$-vertex traceable graph is $n-1$. If we restrict our attention to traceable graphs which have additional properties, e.g. bounded neighborhood diversity, we obtain a larger bound on the number of edges. 
    More precisely, we consider several structural graph parameters and ask the following question: \textit{What is the minimum number of edges an $n$-vertex graph has to have if it is traceable and has a bounded parameter $d$?}

    We show the following tight upper and lower bounds:
    \begin{itemize}[leftmargin = *, rightmargin = 20em]
        \item quadratic for the class of graphs of bounded neighborhood diversity, bounded size of maximum induced matching, or bounded cluster vertex deletion number;
        \item $n\log n$ for the class of cographs or, more generaly, bounded modular-width, and for the class of bounded distance to cograph; and
        \item sligthly superlinear for the class of bounded shrub-depth.
    \end{itemize}
}

\keywords{traceable graphs, Hamiltonian path, density, modular-width, shrub-depth, cographs, neighborhood diversity, maximum induced matching}

\jnlcitation{\cname{%
\author{Taylor M.},
\author{Lauritzen P},
\author{Erath C}, and
\author{Mittal R}}.
\ctitle{On simplifying ‘incremental remap’-based transport schemes.} \cjournal{\it J Comput Phys.} \cvol{2021;00(00):1--18}.}

\maketitle

\section{Introduction}

Extremal combinatorics is an area of discrete mathematics concerned with determining the maximum or minimum size of a collection of finite objects that satisfies certain properties. In graph theory, a  classical result of Turán~\cite{turan1941extremal} states a tight bound on the number of edges a graph can contain while avoiding the complete graph as a subgraph. An extension of this result is the Erdős–Stone theorem which proves similar tight bounds for non-complete graphs~\cite{ErdosStone1946}.

Another direction one can consider is the opposite question: \emph{What is the \textbf{minimum} number of edges a graph has to have while possesing a certain property?} As a trivial example, every connected graph has to contain at least $n-1$ edges and this bound is tight, i.e., there are connected graphs with $n-1$ edges, namely trees. A slightly more interesting notion is the one of \emph{saturation number} introduced by Erdős, Hajnal and Moon in 1964~\cite{ErdosHajnalMoon1964}, which asks, for a given fixed graph $H$ and positive integer $n$: \emph{What is the minimum number of edges an $n$-vertex graph can have such that it does not contain $H$ but adding any edge to it will create a copy of $H$?}. Since then, there have been many results and generalizations of this notion~\cite{Bollobas1968onsaturatedgraphs,FanW15satlinforests,FaudreeFGJ09saturationtrees,FaudreeG13satnearlycomplete,FerraraJMTW12saturationsubdivisions,FurediK13cyclesat,KaszonyiTuza1986saturatedgr,LvHL23satdisjointstars,SullivanW17sattripartite}.

In our setting, we lay two conflicting properties against each other and ask what is the minimum number of edges an $n$-vertex graph can have. On the one hand, we require that the graph possesses some kind of specific structure. Typically, a class capturing a structure of graphs which is of interest for us (e.g., bounded neighborhood diversity, see \Cref{sec:preliminaries} for a formal definition) contains cliques and is closed under complementation. On the other hand, to avoid graphs with too few edges, we require a conflicting property of the graph -- in our case we require that the graph is traceable, i.e., it contains a Hamiltonian path. 

Note that we aim for superlinear bounds for the number of edges. It is immediate that a traceable graph contains $\Omega(n)$ edges as it contains at least $n-1$ edges due to the Hamiltonian path itself. An interesting question is whether it is necessarily \emph{forced} to contain, e.g., $\Omega(n\log n)$ or $\Omega(n^{3/2})$ edges by employing the specific structure.

As a warmup example, consider the class of \emph{cluster graphs}. A graph is a \emph{cluster graph} if every connected component is a clique. Such class clearly contains the edgeless graph on any number of vertices and it also contains cliques, so the number of edges in such graphs can attain both extreme cases: $0$ and $\binom{n}{2}$. However, if we also require the cluster graph to contain a Hamiltonian path, we obtain (up to isomorphism) one graph on $n$ vertices with such properties, namely the complete graph.

For a slightly less trivial example, consider the class of cographs. Again, an edgeless graph is a cograph, but if we require a cograph to also contain a Hamiltonian path, it necessarily contains $\Omega(n\log n)$ edges (as we prove in \Cref{thm:lower_bound_cograph}). On the other hand, this bound is tight as there exists an infinite family of traceable cographs with $O(n\log n)$ edges (\Cref{thm:upper_bound_cograph}).

For some classes, there is a trivial bound, since they contain finitely many traceable graphs. For example, if we consider traceable graphs of bounded tree-depth, 
 there are finitely many such graphs as any such graph can contain at most $2^{\td}$ vertices~\cite{SPARSITY}. More examples include graphs of bounded vertex cover number (such graphs also have bounded tree-depth). 

The problem also becomes trivial if the class contains $n$-vertex paths for infinitely many $n$, such as the class of planar graphs, interval graphs, or graphs of bounded max-leaf number, treewidth, or clique-width (which generalizes the classes of bounded shrub-depth and bounded modular-width). Similarly, by adding a universal vertex to a path, we obtain a traceable graph with $\Theta(n)$ edges with bounded domination number and bounded diameter, giving no hope for a superlinear bound.

\subparagraph*{Our Contribution.}
We show tight bounds on the number of edges in a graph $G$, given that it contains a Hamiltonian path and belongs to a specific graph class. In particular, we give tight results for the following graph classes (see \Cref{tab:density_results} and \Cref{fig:results_hierarchy}): class of graphs with bounded distance to cocluster graphs, class of graphs with bounded maximum induced matching, class of graphs with bounded neighborhood diversity, class of cographs, class of graphs with bounded modular-width, class of graphs with bounded shrub-depth, class of graphs with bounded cluster vertex deletion number, and class of graphs with bounded distance to cograph.

\begin{table}[th!]
    \centering
\begin{tabular}{|c|c|c|c|}
\toprule
     graph class & bound & lower bound & upper bound \\\midrule
     bounded distance to cocluster & $\Theta \left(n^2\right)$ & Thm. \ref{thm:lower_bound_dist_to_cocluster} & trivial\\\hline
     bounded maximum induced matching & $\Theta\left(\frac{n^2}{d}\right)$ & Thm. \ref{thm:lower_bound_maximum_induced_matching} & Cor. \ref{cor:upper_bound_maximum_induced_matching}\\\hline
     bounded neighborhood diversity & $\Theta\left(\frac{n^2}{d}\right)$ & Thm. \ref{thm:lower_bound_nd} & Cor. \ref{cor:upper_bound_nd} \\\hline
     cograph & $\Theta(n\log n)$ & Thm. \ref{thm:lower_bound_cograph} & Thm. \ref{thm:upper_bound_cograph} \\\hline
     bounded modular-width & $\Theta(n\log_d n)$ & Thm. \ref{thm:lower_bound_modular_width} & Thm. \ref{thm:upper_bound_modular_width}\\\hline
     bounded shrub-depth & $\Theta(n^{1+\frac{1}{2^d - 1}})$ & Thm. \ref{thm:lower_bound_sd} & Thm. \ref{thm:upper_bound_sd}\\\hline
     bounded cluster vertex deletion number & $\Theta\left(\frac{n^2}{d}\right)$& Cor. \ref{cor:lower_bound_cvdn} & Cor. \ref{cor:upper_bound_cvdn} \\\hline
     bounded distance to cograph & $\Theta(n\log \frac{n}{d})$ & Cor. \ref{cor:lower_bound_dtcog} & Cor. \ref{cor:upper_bound_dtcog} \\
     \bottomrule
\end{tabular}
    \caption{Overview of our results. First column indicates the considered graph class. Second column describes the asymptotic bound on the number of edges given that the graph is from the given class and is traceable. If the class is described by a parameter, then $d$ is the corresponding upper bound on the parameter. Third and fourth column indicate a reference to the theorem, where the corresponding lower or upper bound is proven for the given class, respectively. 
    }
    \label{tab:density_results}
\end{table}
\begin{figure}
    \centering
		\begin{tikzpicture}
			\tikzstyle{box} = [draw=none,
			fill=gray!10,
			minimum height=2em,
			minimum width=3.5em,
			align=center,
            rounded corners,
			draw=black]
			\tikzstyle{proven} = [fill=green!15]
			\tikzstyle{proving} = [fill=red!15]
			\tikzstyle{txt}=[draw=none,fill=none]
			
			\def\w{2.8}
			\def\h{1.3}
			\node[box,fill=red!15] (clique) at (\w*0,-\h*3) { Clique};
			\node[box] (cluster) at (-\w*1,-\h*2){Cluster};
			\node[box] (cocluster) at (-\w*2,-\h*2){Cocluster};
            \node[box] (dtococluster) at (-\w*2.5,-\h*1) {bnd distance to Cocluster};
			\node[box,fill=red!15] (dtc) at (\w*1,-\h*2){bnd distance to Clique};
			\node[box] (cvdn) at (\w*0,\h*2){bnd cvdn};
			\node[box] (nd) at (\w*2,\h*2){bnd nd};
			\node[box] (mim) at (\w*1,\h*2){bnd mim};
			\node[box,fill=red!15] (eccn) at (\w*1.5,\h*0){bnd \eccn};
            \node[box,fill=red!15] (vccn) at (\w*1,\h*1){bnd \vccn};

			\node[box] (sd) at (\w*2,\h*3.8){bnd shrub-depth};
			
			\node[box] (dtcog) at (-\w*2,\h*6){bnd distance to Cograph};
			
			\node[box] (cograph) at (-\w*1,\h*5){Cograph};
			
			\node[box] (mw) at (\w*0,\h*6){bnd mw};

            \draw[->] (dtococluster) to[bend left] (sd);
            \draw[->] (cvdn) to (sd);
            \draw[->] (cocluster) to (dtococluster);
            \draw[->] (dtococluster) to (dtcog);
            \draw[->] (clique) to (cocluster);
            \draw[->] (cocluster) to (cograph);
			\draw[->] (cluster) to (cograph);
			\draw[->] (clique) to (cluster);
			\draw[->] (cluster) to (cvdn);
			\draw[->] (clique) to (dtc);
			\draw[->] (dtc) to (cvdn);
			
			\draw[->] (dtc) to (eccn);
			
			\draw[->] (eccn) to (vccn);
            \draw[->] (eccn) to (nd);
            \draw[->] (vccn) to (mim);
			
			\draw[->] (nd) to (sd);
			
			\draw[->] (cvdn) to (dtcog);
			
			\draw[->] (cograph) to (dtcog);
			
			\draw[->] (cograph) to (mw);
			
			\draw[->] (nd) to (mw);
			
			\draw[dashed] (-\w*3.1,\h*4.5) to (\w*3,\h*4.5);

			\draw[dashed] (-\w*3.1,\h*3) to (\w*3,\h*3);

            \draw[dashed] (-\w*3.1,-\h*0.5) to (\w*3,-\h*0.5);

            \node[txt] at (7.5,-1){\large $n^2$};
            
			\node[txt] at (7.5,3){\large $\frac{n^2}{d}$};
			
			\node[txt] at (7.5,5){\large $n^{1+\varepsilon}$};
			
			\node[txt] at (7.5,7){\large $n\log n$};
			
		\end{tikzpicture}
    \caption{Graphical overview of our results. Each box represents a graph class and an arrow from $\mathcal{G}$ to $\mathcal{H}$ means $\mathcal{G}\subseteq \mathcal{H}$. Therefore, lower bounds propagate down, whereas upper bounds propagate up. The bounds on the right represent the asymptotic number of edges given that a graph $G$ is from the given class and contains a Hamiltonian path.    
    Graphs from graph classes highlighted in light red contain such number of edges even without a Hamiltonian Path. All the other graph classes contain for any $n$ the edgeless graph on $n$ vertices, hence the presence of the Hamiltonian Path is necessary to obtain these bounds. Precise bounds can also be found in \Cref{tab:density_results}.}
    \label{fig:results_hierarchy}
\end{figure}
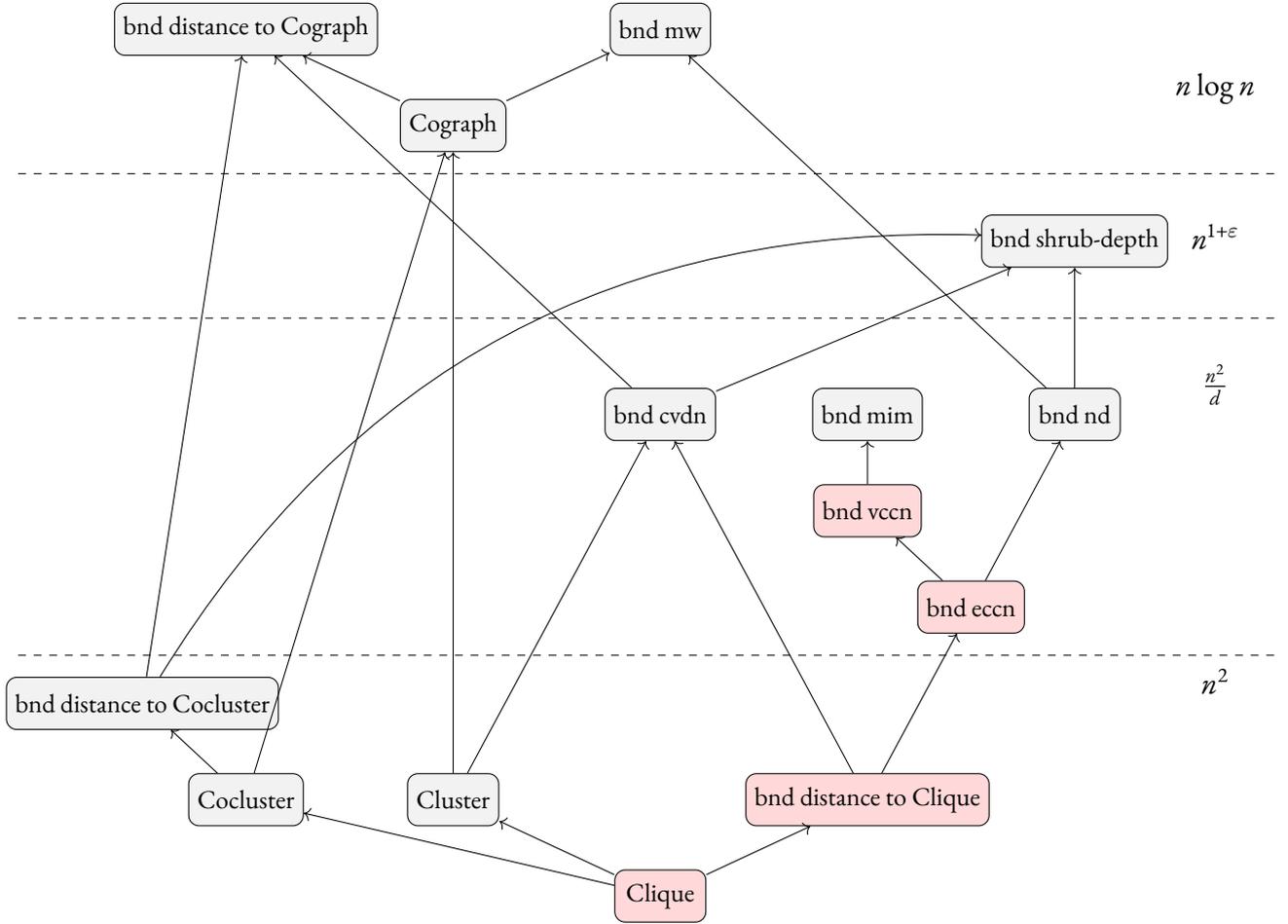

\section{Preliminaries}\label{sec:preliminaries}
We consider simple undirected graphs. A $v_1$-$v_k$ path is a sequence
$(v_1,e_1,v_2,\ldots,e_{k-1},v_k)$
where vertices and edges alternate and $e_i=\{v_i,v_{i+1}\}$. A Hamiltonian path in $G$ is a path containing all vertices of the graph. Graph $G$ is \emph{traceable} if it contains a Hamiltonian path.
A \emph{matching} in $G$ is $1$-regular subgraph of $G$. An \emph{induced} matching in $G$ is a matching that is also an induced subgraph of $G$. The size of a  maximum induced matching of $G$ is denoted $\mim(G)$ and is equal to the number of edges in the matching. A graph $G$ is a \emph{cluster} if every connected component of $G$ is a clique. A graph is a \emph{cocluster} if its complement is cluster, in other words, if it is a complete multipartite graph. Graph $G$ is said to be $H$-free if it does not contain $H$ as an induced subgraph. We use the following well-known result of Turán.
\begin{theorem}[Turán~\cite{turan1941extremal}]
    Any $K_{d+1}$-free graph $G$ with $n$ vertices has at most $\left(1-\frac{1}{d}\right)\cdot\frac{n^2}{2}$ edges.
\end{theorem}

We also use the following well-known inequality by Jensen.

\begin{lemma}[Jensen's Inequality, Jensen~\cite{Jensen1906}]
    Let $\varphi\colon \mathbb{R}\to \mathbb{R}$ be a convex function, then for any $x_1,x_2,\ldots,x_\ell\in\mathbb{R}$ we have ${\sum_{i=1}^\ell \varphi(x_i)\geq \ell \cdot \varphi\left(\frac{\sum_{i=1}^\ell x_i}{\ell}\right)}$.
\end{lemma}

The \emph{vertex (edge) clique cover number} of a graph $G$, denoted $\vccn(G)$ ($\eccn(G)$), is the minimum number of complete subgraphs of~$G$ such that every vertex (edge) is contained in one of the subgraphs. 
If $G$ has no isolated vertices, then every edge clique cover is also a vertex clique cover, hence $\vccn(G)\leq \eccn(G)$ for every such $G$.

The \emph{neighborhood diversity} of a graph $G$, denoted $\nd(G)$, is the smallest integer $d$ such that there exists a partition of $V$ into $d$ sets $V_1,\ldots,V_d$ such that for any $i\in[d]$ and $u,v\in V_i$ it holds $N(u)\setminus \{v\}=N(v)\setminus \{u\}$. In other words, the sets $V_i$ are either independent or cliques and for any two distinct $V_i,V_j$ either we have for every $v_i\in V_i,v_j\in V_j$ that $\{v_i,v_j\}\in E(G)$ or for every $v_i\in V_i,v_j\in V_j$ that $\{v_i,v_j\}\notin E(G)$.

\subparagraph*{Modular-width}
The following definition is from~\cite{GajarskyLO2013}.
Consider an algebraic expression $A$ that uses the following operations:
\begin{enumerate}[label=(O\arabic*)]
    \item create an isolated vertex;\label{itm:modular_width_O1}
    \item take the \emph{disjoint union} of graphs $G_1,G_2$, denoted by $G_1\oplus G_2$, which is the graph with vertex set $V(G_1)\cup V(G_2)$ and edge set $E(G_1)\cup E(G_2)$;\label{itm:modular_width_O2}
    \item take the \emph{complete join} of $2$ graphs $G_1$ and $G_2$, denoted by $G_1\otimes G_2$, which is the graph with vertex set $V(G_1)\cup V(G_2)$ and edge set $E(G_1)\cup E(G_2)\cup \{\{v,w\}\mid v\in V(G_1)\wedge w\in V(G_2)\}$;\label{itm:modular_width_O3}
    \item for graphs $G_1,\ldots, G_n$ and a pattern graph $G$ with vertices $v_1,\ldots,v_n$ perform the \emph{substitution} of the vertices of $G$ by the graphs $G_1,\ldots, G_n$, denoted by $G(G_1,\ldots, G_n)$, which is the graph with vertex set $\bigcup_{i=1}^n V(G_i)$ and edge set $\bigcup_{i=1}^n E(G_i)\cup \{\{u,v\}\mid u \in V(G_i)\wedge v \in V(G_j)\wedge \{v_i,v_j\}\in E(G)\}$. Hence, $G(G_1,\ldots, G_n)$ is obtained from $G$ by replacing every vertex $v_i\in V(G)$ with the graph $G_i$ and adding all edges between vertices of a graph $G_i$ and the vertices of a graph $G_j$ whenever $\{v_i,v_j\}\in E(G)$.\label{itm:modular_width_O4}
\end{enumerate}
The \emph{width} of the expression $A$ is the maximum number of vertices of a pattern graph used by any occurrence of the operation \ref{itm:modular_width_O4} in~$A$ (or $0$ if \ref{itm:modular_width_O4} does not occur in $A$). The \emph{modular-width} of a graph~$G$, denoted $\operatorname{mw}(G)$, is the smallest integer $m$ such that $G$ can be obtained from such an algebraic expression of width at most $m$. Note that the operations \ref{itm:modular_width_O2} and \ref{itm:modular_width_O3} can be seen as a special case of \ref{itm:modular_width_O4} with graphs $K_2$, resp. $\overline{K_2}$. 
The graphs $G_1,\ldots, G_n$ which become the subgraphs of $G(G_1,\ldots, G_n)$ are referred to as \emph{modular partition} of $G(G_1,\ldots, G_n)$.

A graph is a \emph{cograph} if and only if it can be made by operations \ref{itm:modular_width_O1}, \ref{itm:modular_width_O2}, and \ref{itm:modular_width_O3}. Or, equivalently, a graph is a cograph if and only if it has modular-width~$0$.

\subparagraph*{Shrub-depth} To define shrub-depth we first need the following concept.
\begin{definition}[Tree-model~\cite{Ganian2017ShrubdepthCH}]
    Let $m$ and $d$ be non-negative integers. A \emph{tree-model of $m$ colours and depth $d$} (or a \emph{$(d,m)$-tree-model} for short) for a graph $G$ is a pair $(T,S)$ of a rooted tree $T$ (of height $d$) and a set $S\subseteq [m]^2\times [d]$ (called a \emph{signature} of the tree-model) such that
    \begin{enumerate}
        \item the length of each root-to-leaf path in $T$ is exactly $d$,
        \item the set of leaves of $T$ is exactly the set $V(G)$ of vertices of $G$,
        \item each leaf of $T$ is assigned one of the colours in $[m]$, and
        \item for any $i,j,\ell$ it holds that $(i,j,\ell)\in S\Leftrightarrow (j,i,\ell)\in S$ (symmetry in the colours), and 
        \item for any two vertices $u,v\in V(G)$ and any $i,j,\ell$ such that $u$ is coloured $i$ and $v$ is coloured $j$ and the distance between $u,v$ in $T$ is $2\ell$, the edge $\{u,v\}$ exists in $G$ if and only if $(i,j,\ell)\in S$.
    \end{enumerate}
\end{definition}
\begin{definition}[Shrub-depth]
    A class $\mathcal{G}$ of graphs has \emph{shrub-depth} at most $d$ if there exists~$m$ such that each $G\in\mathcal{G}$ admits a $(d,m)$-tree-model.
\end{definition}

\subparagraph*{Distance to $\Pi$}
Let $\Pi$ be a graph class. A set $S\subseteq V$ is a \emph{modulator} to $\Pi$ if $G\setminus S\in \Pi$. For a graph $G$, the \emph{distance to $\Pi$ of $G$} is the size of the smallest modulator to $\Pi$ and we denote it $\operatorname{dist}_{\Pi}(G)$. The \emph{cluster vertex deletion number} (\cvdn) is the distance to cluster graphs.

\section{Results}\label{sec:results}
In this section we prove our results. In \Cref{subsec:quadratic_bounds} we prove the quadratic bounds for graphs of bounded distance to cocluster, graphs of bounded maximum induced matching, and for graphs of bounded neighborhood diversity. In \Cref{subsec:nlogn_bounds} we focus on the $n\log n$ bounds for the classes of cographs and graphs of bounded modular-width. In \Cref{subsec:shrub_depth} we prove the $n^{1+\varepsilon}$ bound for graphs of bounded shrub-depth. Finally, in \Cref{subsec:distance_to_pi} we prove the bounds for classes characterized by having bounded distance to $\Pi$ for some property $\Pi$.

\subsection{$n^2$ bounds}\label{subsec:quadratic_bounds}
\begin{theorem}\label{thm:lower_bound_dist_to_cocluster}
    Let $\Pi$ be the class of coclusters, $d\geq 0$ arbitrary, and let $G$ be a traceable graph with $\operatorname{dist}_{\Pi}(G)\leq d$ and $n\geq 4d+4$ vertices, then $G$ contains at least $\frac{n^2}{16}$ edges.
\end{theorem}
\begin{proof}
    Let $G=(V,E)$ and let $S\subseteq V$ be a modulator to $\Pi$. $G\setminus S$ is a cocluster, in other words, a complete $k$-partite graph. Let $V\setminus S=V_1\cup \cdots \cup V_k$ be the parts of $G\setminus S$. We show that no part can be too big. In particular, for every $i\in[k]$ we have $|V_i|\leq \frac{n}{2}+1$. To see this, assume for contradiction that $|V_i|>\frac{n}{2}+1$ for some $i$ and consider the Hamiltonian path $P$ in $G$. Every vertex $x\in V_i$ has a unique successor on $P$ (except possibly one vertex in $V_i$ which is the endpoint of $P$), so there are at least $\frac{n}{2}$ vertices outside $V_i$, which implies there are at least $|V_i|+\frac{n}{2}=n+1>n$ vertices in $G$---a contradiction.
    Since $G\setminus S$ is complete $k$-partite and each $V_i$ is of size at most $\frac{n}{2}+1$ it follows that every vertex of $V\setminus S$ is incident to at least $n-d-(\frac{n}{2}+1)=\frac{n}{2}-d-1$ other vertices inside $V\setminus S$. We obtain the following bound:
    \[
    |E(G)|\geq |E(G\setminus S)|=\frac{1}{2}\sum_{v\in V\setminus S}\deg_{G\setminus S}(v)\geq \frac{1}{2}(n-d)\cdot \left(\frac{n}{2}-d-1\right)\geq \frac{n^2}{16}.
    \]
    We used the inequalities $n-d\geq \frac{n}{2}$ and $\frac{n}{2}-d-1\geq \frac{n}{4}$ valid for $n\geq 4d+4$.
\end{proof}

\begin{theorem}\label{thm:lower_bound_maximum_induced_matching}
    Let $G$ be a traceable graph on $n$ vertices and $d=\mim(G)$ be the size of a maximum induced matching of $G$. Then $G$ contains at least $\frac{1}{8}\cdot \frac{n^2}{d}$ edges.
\end{theorem}
\begin{proof}
    Let $P$ be a Hamiltonian path of $G$. 
    Consider an auxiliary graph $H$ with $V(H)=E(P)$ and $\{e,f\}\in E(H)$ if and only if $e,f$ induce an induced matching in $G$, i.e., $e \cap f = \emptyset$ and $E(G[e \cup f]) = \{e,f\}$. 
    Notice that a $K_{d+1}$ in $H$ would correspond to an induced matching of size $d+1$ in $G$.
    Consider the complement $\overline{H}$ and notice that each edge of $\overline{H}$ corresponds to either two consecutive edges on~$P$ (these do not form an induced matching) or at least one edge in $E(G)\setminus E(P)$. 
    Picking any of the edges if there are more of them provides a mapping from $E(\overline{H})\setminus \{\{e,f\}\mid\text{$e$, $f$ are consecutive on $P$}\}$ to $E(G)\setminus E(P)$ which maps at most $4$ edges of $E(\overline{H})$ to one edge of $E(G)\setminus E(P)$.

    We lower bound $|E(\overline{H})|$. 
    As $H$ is $K_{d+1}$-free, by Turán's theorem, we have $|E(H)|\le (1-\frac{1}{d})\frac{n^2}{2}$.
    Therefore, $|E(\overline{H})|\geq \binom{n}{2}-(1-\frac{1}{d})\frac{n^2}{2}\geq \frac{n^2}{2d} - \frac{n}2$. 
    By the above argumentation, 
    we have that $\frac{n^2}{2d} - \frac{n}2 -(n-2) \le |E(\overline{H})|-(n-2) \leq 4(|E(G)|-(n-1))$, which entails $|E(G)|\geq \frac14 \cdot \left(\frac{n^2}{2d} - \frac{n}2 -(n-2)\right)+n-1 \ge \frac{n^2}{8d}$ which we wanted to show.
\end{proof}

\begin{theorem}\label{thm:lower_bound_nd}
    Let $G$ be a traceable graph with neighborhood diversity $\nd(G)$ and $n\geq 2$ vertices. Then $G$ contains at least $\frac{1}{4}\cdot \frac{n^2}{\nd(G)}$ edges.
\end{theorem}
\begin{proof}
First, note that $G$ has at least $n-1$ edges, and if $n \le 2\nd(G)$, then 
\[n-1 \ge \frac{n}2 \ge \frac{n\cdot 2\nd(G)}{4\nd(G)} \ge \frac{n^2}{4\nd(G)}\]
so the claim holds.
Therefore we assume that $G$ has at least $2\nd(G)$ vertices.
Let $G$ be a traceable graph with a neighborhood diversity partition of $V(G)$ into $d=\nd(G)$ sets $V_1,\ldots,V_d$ and let $n_i$ denote the size of~$V_i$ for each $i \in [d]$.
We bound the number of edges incident to each $V_i$ separately depending on whether $V_i$ induces a clique or an independent set.
If $V_i$ induces a clique then there are exactly $\binom{n_i}{2}$ edges with both endpoints in~$V_i$.
Otherwise, $V_i$ induces an independent set and no two of its vertices can appear next to each other in the Hamiltonian path.
Therefore, there are at least $n_i-1$ vertices outside of $V_i$ with a neighbor in~$V_i$.
Due to the properties of neighborhood diversity, all these vertices are in fact adjacent to the whole set~$V_i$ and it follows that $V_i$ is incident to at least $n_i (n_i-1)$ edges.
Notice that in this second case we may consider each edge twice, once from each endpoint.
Accounting for this overcounting, we obtain
\[
  |E(G)| 
  \ge \sum_{\text{$V_i$ clique}} \binom{n_i}{2} + \sum_{\text{$V_i$ independent}} \frac{n_i (n_i - 1)}{2}  
  =\sum_{i=1}^d \frac{n_i (n_i - 1)}{2}
  \geq d\cdot \frac{\frac{n}{d} (\frac{n}{d}-1)}{2}  
  = \frac{n(n-d)}{2d}
  \geq \frac{n^2}{4d}.
\]  
The second inequality follows by applying Jensen's inequality to the convex function $x \mapsto \frac{x (x-1)}{2}$. The last inequality follows from the fact that $n-d\geq \frac{n}{2}$ because $n\geq 2d$.
\end{proof}

\begin{theorem}\label{thm:upper_bound_eccn}
    For each $d\geq 1$ and infinitely many $n$ there is a traceable graph on $n$ vertices with edge clique cover number at most $d$ and at most $\frac{2n^2}{d}$ edges.
\end{theorem}
\begin{proof}
    For any $k\geq 2$ we construct a graph $G$ on $n=kd-d+1$ vertices as follows. Construct $d$ cliques $K^1,K^2,\ldots,K^d$ of size $k$ and identify one vertex of $K^1$ with a vertex of $K^2$, for every $i\geq 2$ identify one vertex of $K^i$ (an arbitrary one not identified with a vertex of $K^{i-1}$) and identify it with a vertex of $K^{i+1}$. The graph clearly has $n$ vertices and the cliques form an edge clique cover for $G$. The number of edges of $G$ is $d\binom{k}{2}\leq d\frac{k^2}{2}\leq \frac{2n^2}{d}$,
    as claimed. We used the fact that $kd=n+d-1\leq 2n$, for $n$ large enough  so $k\leq \frac{2n}{d}$.
\end{proof}

\begin{corollary}\label{cor:upper_bound_maximum_induced_matching}
    For any $d\geq 1$ and infinitely many $n$ there is a traceable graph $G$ with at most $\frac{2n^2}{d}$ edges and maximum induced matching of size $d$.
\end{corollary}
\begin{proof}
    We use \Cref{thm:upper_bound_eccn}. The constructed graph has $\mim \leq d$. This follows from the fact that $\mim \leq \vccn \leq \eccn$. To see this, suppose that $\mim > \vccn$ and consider the vertex clique cover of $G$. Since $\mim > \vccn$, by the Pigeonhole principle, there is a clique containing two edges of the induced matching, contradiction. $\vccn \leq \eccn$ follows from the fact that $G$ contains no isolated vertices.
\end{proof}

\begin{corollary}\label{cor:upper_bound_nd}
    For each $d\geq 1$ and infinitely many $n$ there is a traceable graph on $n$ vertices with neighborhood diversity at most $2d-1$ and at most $\frac{2n^2}{d}$ edges.
\end{corollary}
\begin{proof}
    Use the proof of \Cref{thm:upper_bound_eccn} and notice that it has neighborhood diversity at most $2d-1$. The neighborhood diversity partition is made of the $d$ cliques $K^1,K^2,\ldots,K^d$ with vertices in the intersection $V(K^i)\cap V(K^{i+1})$ removed and the $d-1$ vertices that are in the intersection $V(K^i)\cap V(K^{i+1})$ for $i\in[d-1]$. This is a neighborhood diversity partition of $G$ of size $2d-1$.
\end{proof}

\subsection{$n\log n$ bounds}\label{subsec:nlogn_bounds}

\begin{theorem}\label{thm:lower_bound_cograph}
    Let $G$ be a traceable cograph. Then $G$ contains at least $\frac{1}{4}\cdot n\log_2n$ edges.
\end{theorem}
\begin{proof}
    We proceed by induction on $n$. The claim is obvious for $n\leq 4$. Let $G$ be traceable cograph on $n\geq 5$ vertices. Since $G$ contains a Hamiltonian path it is necessarily a join of two smaller cographs $G_1$ and $G_2$. Denote $|V(G_1)|=k$ and $|V(G_2)|=n-k$ and assume without loss of generality that $k\leq \frac{n}{2}$. 
    As $n-k \ge \frac{n}2$, we have at least $k(n-k) \ge \frac{nk}2$ edges between the parts.    
    We split the Hamiltonian path into $\ell$ subpaths $P_1,\ldots,P_\ell$ by removing edges with one endpoint in $G_1$ and other endpoint in $G_2$. Clearly $2\leq \ell \leq 2k+1$.
    The total number of edges is $|E(G)|\geq \frac{kn}{2}+\sum_{i=1}^\ell|E(P_i)|$. By the induction hypothesis, since the paths $P_i$ induce a Hamiltonian path in the subgraph induced by $V(P_i)$, each segment induces at least $\frac{1}{4}|P_i|\log |P_i|$ edges. Hence
    \begin{align*}
        |E(G)|&\geq\frac{kn}{2}+\frac{1}{4}\sum_{i=1}^\ell |P_i|\log |P_i| \geq \\ &\geq \frac{kn}{2}+\frac{1}{4}\ell \frac{n}{\ell}\log\frac{n}{\ell}\geq\\&\geq \frac{kn}{2}+\frac{1}{4}n\log \left(\frac{n}{2k+1}\right) =\\&=\frac12nk+\frac14n\log n-\frac14 n \log (2k+1)
    \end{align*}
    where the last two inequalities are due to Jensen's inequality and the fact that $\ell \leq 2k+1$. Since $\frac{k}{2}-\frac{1}{4}\log (2k+1)\geq 0$ for any $k\geq 1$, the last term is lower bounded by $\frac{1}{4}n\log n$ as desired.
\end{proof}

\begin{theorem}\label{thm:upper_bound_cograph}
    For infinitely many $n$ there is a traceable cograph on $n$ vertices with at most $n\log_2 n$ edges.
\end{theorem}
\begin{proof}
    For every $k$ we construct a traceable cograph with $n=2^k-1$ vertices and at most $2^kk\leq n\log_2 n$ edges. 
    We proceed by induction on~$k$. For $k=1$ we take $P_1$ and for $k=2$ the path $P_3$.

    For the inductive step, take a disjoint union of two traceable cographs with $2^k-1$ vertices and at most $2^kk$ edges each. 
    Add a new vertex and perform a join with the union of the two graphs on one side and the new vertex on the other side. The new graph has $2\cdot (2^k-1)+1=2^{k+1}-1$ vertices. Clearly it is traceable and it has at most $2\cdot 2^kk+2(2^k-1)\leq 2^{k+1}(k+1)$ edges.
\end{proof}

\begin{theorem}\label{thm:lower_bound_modular_width}
    Let $G$ be a traceable graph and let $d=\mw(G)$ be the modular-width of $G$. Then $G$ has at least $\frac{1}{4}\cdot n\log_d n$ edges.
\end{theorem}
\begin{proof}
    We proceed by induction on $n$. The claim is obvious for $n\leq 4$. Let $G$ be a graph with modular-width~$d$ on $n\geq 5$ vertices and let $P$ be its arbitrary Hamiltonian path.
    Let $M_1, \dots, M_p$ be a modular partition of~$G$ where $p \le d$.
    For each $i \in [p]$, we let $n_i$ denote the size of $M_i$ and we let $\ell_i$ be the number of disjoint paths $P^i_1, \dots, P^i_{\ell_i}$ obtained by restricting~$P$ to the module~$M_i$.
    
    For each $i \in [p]$, we separately bound the edges with both endpoints inside~$M_i$ and the edges with exactly one endpoint in~$M_i$. 
    By the induction hypothesis, the graph induced by the segment~$P^i_j$ has at least $\frac14|P^i_j|\log_d|P^i_j|$ edges since it is itself a traceable graph with modular-width at most~$d$.
    Therefore, the number of edges with both endpoints in~$M_i$ is at least 
    \[\sum_{j=1}^{\ell_i} \frac14|P^i_j|\log_d|P^i_j| \ge \frac14 n_i\log_d \frac{n_i}{\ell_i}, \]
    where the inequality follows by using Jensen's inequality on the convex function $x \mapsto x \log_d x$.
    
    On the other hand, there are at least $\ell_i-1$ vertices outside of~$M_i$ with neighbors in~$M_i$ since~$P$ must contain at least one vertex outside of~$M_i$ between the segments $P^i_j$ and $P^i_{j+1}$ for each $j \in [\ell_i - 1]$.
    Moreover, in the special case when $\ell_i = 1$, there must still exist at least one such vertex outside of~$M_i$.
    By the properties of modules, these vertices are adjacent to the whole module~$M_i$ and thus there are at least $n_i \ell_i^-$ edges with exactly one endpoint in~$M_i$ where $\ell_i^- = \max(1, \ell_i - 1)$.
    
    Summing these together and adjusting for overcounting, we obtain
    \begin{align*}
      |E(G)| &\ge \sum_{i=1}^p \frac14 n_i\log_d \frac{n_i}{\ell_i} + \sum_{i=1}^p \frac{n_i \ell_i^-}{2} =\\&= \sum_{i=1}^p \frac14 n_i\log_d \frac{d \cdot n_i}{d \cdot \ell_i} + \sum_{i=1}^p \frac{n_i \ell_i^-}{2} =\\&=\sum_{i=1}^p \frac14 n_i \log_d(d \cdot n_i) + \sum_{i=1}^p \frac14 n_i \cdot \left(2\ell_i^- - \log_d(d\ell_i)\right)
    \end{align*}
    where the second term on the right hand side is always non-negative since $2(x-1) - \log_d(dx) \ge 0$ for every $x \ge 2$ and $2\ell_i^- - \log_d(d\ell_i) = 1 \ge 0$ in the special case when $\ell_i = 1$.
    
    Finally, it suffices to apply Jensen's inequality one more time to obtain 
    \[|E(G)| \ge  \sum_{i=1}^p \frac14 n_i \log_d(d \cdot n_i) \ge  \frac14 p \cdot \frac{n}{p}\log_d\left(\frac{d n}{p}\right) \ge \frac14 n \log_d n,\]
    where the last inequality follows from $p \le d$.
\end{proof}

\begin{theorem}\label{thm:upper_bound_modular_width}
    For each $d\geq 2$ and infinitely many $n$ there is a traceable graph on $n$ vertices with modular-width at most $2d$ and at most $2n\log_d n$ edges.
\end{theorem}
\begin{proof}
   For every $k \ge 1$ we construct a graph $G_k$ with $n=\sum_{i=1}^k d^i$ vertices, modular-width at most $2d$ and at most $2k \cdot \sum_{i=1}^k d^i \le 2n\log_d n$ edges, where the last inequality follows from $\log_d \sum_{i=1}^k d^i \ge \log_d d^k = k$.
    
    We proceed by induction on $k$. 
    For $k=1$ we take $G_1$ to be the path on $d$ vertices. 
    Clearly, it has $d$ vertices, modular width at most $2d$ and $d-1 \le 2\cdot 1 \cdot d$ edges.
    Let $k \ge 2$. 
    We create the graph $G_k$ by taking a path on $2d$ vertices $x_1,x_2,\ldots,x_{2d}$ and replacing the vertices $x_1,x_3,x_5,\ldots,x_{2d-1}$ by copies of the graph~$G_{k-1}$. 
    The graph $G_k$ has $d\cdot |V(G_{k-1})| + d = d \cdot \sum_{i=1}^{k-1} d^i + d = \sum_{i=1}^k d^i$ vertices and  modular-width at most $2d$ as claimed.
    The number of edges in $G_k$ is at most \[d \cdot 2(k-1)\sum_{i=1}^{k-1}d^i + (2d-1) \cdot \sum_{i=1}^{k-1}d^i \le 2(k-1)\sum_{i=2}^{k}d^i + 2\cdot \sum_{i=2}^{k}d^i \le 2k\sum_{i=1}^{k}d^i. \qedhere\] 
\end{proof}

\subsection{$n^{1+\varepsilon}$ bound for shrub-depth}\label{subsec:shrub_depth}

\begin{theorem}\label{thm:lower_bound_sd}
    Let $m\geq 2, d\geq 1$ be given integers. Let $G$ be a traceable graph on $n\geq 2$ vertices admitting a $(d,m)$-tree-model. Then $G$ has at least $\frac{1}{16m}\cdot n^{1+\frac{1}{2^d - 1}}$ edges.
\end{theorem}
\begin{proof}
    We proceed by induction on $d$. For $d=1$, $G$ has neighborhood diversity at most $m$ so, by \Cref{thm:lower_bound_nd}, the graph has at least $\frac{n^2}{4m}\geq \frac{n^2}{16m}$ edges.

    For the inductive step, consider the root of the tree-model $r$ and let $T_1,\ldots, T_s$ be subtrees corresponding to the children of $r$, $s\geq 2$. 
    Let $P$ be a Hamiltonian path of $G$.
    Cut the Hamiltonian path $P$ into $\ell\geq 2$ segments of sizes $n_1,n_2,\ldots,n_{\ell}$ by cutting an edge if its endpoints belong to different subtrees among $T_1,\ldots, T_s$. Let $v_1,v_2,\ldots,v_{2\ell}$ be the endpoints of the segments. Note that it might happen that (e.g.) $v_{2i-1}=v_{2i}$ if the segment consists of a single vertex. Consider the right endpoints of every second segment starting from the first one, i.e., the vertices $X=\{v_2,v_6,v_{10},\ldots,\}$.
    For $j\in[m]$ let $X_j$ be the set of vertices in $X$ with label $j$. 
    For each $v_k\in X_j$ note that the vertex $v_{k+1}$ is, by definition, in a different subtree than $v_k$. 
    Partition $X_j$ by the subtrees $T_1,\ldots, T_s$ in which the vertices occur, i.e., $X_j = X_j^1\cup \cdots X_j^s$. 
    We now distinguish two cases.

    \begin{description}
        \item[Case 1] There is $t\in[s]$ such that $|X_j^t|\geq \frac{1}{2}\cdot |X_j|$. In this case, since all successors of vertices in $X_j^t$ are in different subtrees, there is a complete bipartite graph between the vertices in $X_j^t$ (one part) and the set of successors of vertices in $X_j^t$, formally the set $Y=\{v_{k+1} \mid v_k \in X_j^t\}$. Note that since we chose every second segment, we have $Y\cap X_j^t=\emptyset$, hence there are at least $\frac{|X_j|^2}{4}$ edges.

        \item[Case 2] For all $t\in[s]$ we have $|X_j^t|<\frac{1}{2}\cdot |X_j|$. In this case we consider the successors $Y$ of vertices in whole $X_j$, i.e., $Y=\{v_{k+1}\mid v_k\in X_j\}$. By the same argument as in the previous case, $Y\cap X_j=\emptyset$. Consider a vertex $v\in Y$ and suppose that it is in subtree $T_t$. Note that less than $\frac{1}{2}\cdot |X_j|$ vertices are in $T_t$, hence there are at least $|X_j|-\frac{1}{2}|X_j|=\frac{1}{2}|X_j|$ vertices from $X_j$ that are not in $T_t$. By definition of shrub-depth, $v$ is connected to at least this many vertices. Summing for each $v\in Y$ we obtain $\frac{|X_j|^2}{2}$ edges in this case.
    \end{description}

    In both cases we obtain at least $\frac{|X_j|^2}{4}$ edges. By summing over $j$, we obtain at least $\sum_{j=1}^m\frac{|X_j|^2}{4}\geq m \cdot \frac14 \cdot \left(\frac{|X|}{m}\right)^2= \frac{|X|^2}{4m}$ edges, where the inequality is due to Jensen's inequality applied to the convex function $x\mapsto x^2$. Since $|X|=\frac{\ell}{2}$ we obtain at least $\frac{\ell^2}{16m}$ edges between different subtrees from $T_1,\ldots, T_s$. We apply induction hypothesis to the $\ell$ segments and $d-1$ and count the edges inside the subtrees $T_1,\ldots, T_s$. For a segment of length $n_i$ we have at least $\frac{1}{16m}n_i^{1+\frac{1}{2^{d-1} - 1}} = \frac{1}{16m}n_i^{\frac{2^{d-1}}{2^{d-1}-1}}$ edges. The total number of edges is thus at least 

    \begin{equation}\label{eq:sd_bound}
        \frac{\ell^2}{16m}+\sum_{i=1}^{\ell}\frac{1}{16m}n_i^{\frac{2^{d-1}}{2^{d-1}-1}}\geq 
        \frac{1}{16m}\left(\ell^2+\ell\left(\frac{n}{\ell}\right)^{\frac{2^{d-1}}{2^{d-1}-1}}\right)=
        \frac{1}{16m}\left(\ell^2+\frac{n^{\frac{2^{d-1}}{2^{d-1}-1}}}{\ell^{\frac{1}{2^{d-1}-1}}}\right),
    \end{equation}
    where we applied the Jensen's inequality to the convex function $n\mapsto n^{\frac{2^{d-1}}{2^{d-1}-1}}$.
    We lower bound the parenthesis as follows. Note that one summand is increasing in $\ell$ while the other is decreasing in $\ell$, hence the sum is lower bounded by the value attained by any of them when they are equal, that is, when $\ell^2 = \frac{n^{\frac{2^{d-1}}{2^{d-1}-1}}}{\ell^{\frac{1}{2^{d-1}-1}}}$, which is precisely when 
    \[\ell = n^{\frac{\frac{2^{d-1}}{2^{d-1}-1}}{2+\frac{1}{2^{d-1}-1}}}=n^{\frac{2^{d-1}}{2(2^{d-1}-1)+1}}=n^{\frac{2^{(d-1)}}{2^d-1}}.\] By plugging back we lower bound the last expression in (\ref{eq:sd_bound}) by
    $\frac{1}{16m} \cdot n^{\frac{2^d}{2^d-1}} = \frac{1}{16m} \cdot n^{1+\frac{1}{2^d - 1}}$, 
    which was to be shown.
\end{proof}

\begin{theorem}\label{thm:upper_bound_sd}
For any $d$ and infinitely many $n$, there is a traceable graph~$G$ admitting $(d, 2^{d-1})$-tree-model with $n$ vertices and at most $2^{O(d)} \cdot n^{1+\frac{1}{2^d-1}}$ edges.
\end{theorem}
\begin{proof}
We prove by induction on~$d$ that for every $n \ge 1$, there is a traceable graph with at least $n$ vertices and at most $\gamma_d \cdot n^{1+\frac{1}{2^d-1}}=\gamma_d \cdot n^{2^d/(2^d-1)}$ edges that admits a $(d, 2^{d-1})$-tree-model where the constants $\gamma_d$ are given by the recurrence $\gamma_d = 4\cdot \gamma_{d-1}+32$ with initial condition $\gamma_1 = 1$.
Notice that the bound on the number of edges in the statement is with respect to~$n$ even though the graph can contain more than~$n$ vertices.
Therefore, we can restrict such traceable graph to its induced traceable subgraph on exactly~$n$ vertices without violating the edge bound.
This implies the statement since the recurrence solves to $\gamma_d = \frac{1}{12}(35\cdot 4^d - 128) = 2^{O(d)}$.

For $d = 1$, the inequality holds since any graph on $n$ vertices has less than~$n^2$ edges and we can, thus, simply take a clique which admits a trivial $(1,1)$-tree-model.
Let $ d > 1$ and set $n' = \lfloor n^{(2^{d-1} - 1)/(2^d - 1)} \rfloor$.
By induction hypothesis, there is a traceable graph~$G'$ on at least $n'$ vertices with at most $\gamma_{d-1} \cdot (n')^{2^{d-1}/(2^{d-1}-1)}$ edges that admits a $(d-1, 2^{d-2})$-tree-model $(T',S')$.
By plugging in the inequality $n' \le n^{(2^{d-1} - 1)/(2^d - 1)}$, we obtain
\begin{equation}\label{eq:sd-edges-induction}
|E(G')| \le \gamma_{d-1} \cdot (n')^{2^{d-1}/(2^{d-1}-1)} \le \gamma_{d-1} \cdot \left(n^{(2^{d-1} - 1)/(2^d - 1)}\right)^{2^{d-1}/(2^{d-1}-1)} = \gamma_{d-1} \cdot n^{2^{d-1}/(2^d - 1)}.
\end{equation}

Our goal is to take $\lceil n / n' \rceil$ many copies of $G'$ and join them to obtain a traceable graph on at least $n$ vertices.
First, we double in size the set of available colours.
In order to simplify the argument, let us work from now on with colours from the set $\{0,1\} \times [2^{d-2}]$ instead of the set $[2^{d-1}]$.
We define a signature $S \subseteq (\{0,1\} \times [2^{d-2}])^2 \times [d]$.
For every $\ell < d$, the signature contains exactly the tuples that appear in~$S'$, ignoring the first coordinate of the colours, i.e. we have for every $\alpha_1, \alpha_2 \in \{0,1\}$ and $\ell < d$
\begin{equation}\label{eq:sd-relation}
((\alpha_1, i), (\alpha_2, j), \ell) \in S \iff  (i, j, \ell) \in S'.
\end{equation}
At the root level, the signature contains exactly all possible pairs of colours with their first coordinate equal to~1.
That is, we additionally let $S$ contain the tuples $((1, i), (1, j), d)$ for every choice of $i, j \in [2^{d-2}]$.

We now define a $(d-1, 2^{d-1})$-tree-model $(T'',S)$ of $G'$ that allows us to distinguish the endpoints of the Hamiltonian path.
The tree~$T''$ is exactly the same as~$T'$ except for the assignment of colours to its leaves.
Let $f\colon V(G') \to [2^{d-2}]$ be the assignment of colours to vertices of~$G'$ (as leaves) in~$T'$.
Then $v \in V(G')$ receives in~$T''$ the colour $(1, f(v))$ if it is one of the two endpoints of the Hamiltonian path, and $(0, f(v))$ otherwise.
Observe that $(T'', S)$ is indeed a tree model of $G'$ by~\eqref{eq:sd-relation} since $T''$ has depth $d-1$.

We now let $G$ be the graph defined by a tree model $(T,S)$ where $T$ consists of a root with $\lceil n / n' \rceil$ subtrees, each a copy of~$T''$.
Notice that $G$ is traceable graph since the endpoints of all individual Hamiltonian paths in the subtrees are connected at the root level into a clique.
For technical reasons, let us first observe the following upper bound on $\lceil n / n' \rceil$
\begin{equation}\label{eq:sd-num-subtrees}
\left\lceil \frac{n}{n'} \right\rceil \le 2 \cdot \frac{n}{n'} \le 2 \cdot \left(\frac{2 n}{n^{(2^{d-1} - 1)/(2^d - 1)}}\right) \le 4 \cdot n^{2^{d-1}/(2^d -1)}
\end{equation}
where the first inequality holds since $n / n' \ge 1$ and the second follows from the fact that $n^{(2^{d-1} - 1)/(2^d - 1)} > 1$ and thus, $n' = \lfloor n^{(2^{d-1} - 1)/(2^d - 1)} \rfloor\ge n^{(2^{d-1} - 1)/(2^d - 1)} / 2$.

Now we separately bound the number of edges within the copies of $G'$ and the edges introduced at the root level.
There are $\lceil n / n' \rceil$ copies of~$G'$ in total, each with at most $\gamma_{d-1} \cdot n^{2^{d-1}/(2^d - 1)}$ edges by~\eqref{eq:sd-edges-induction}.
The number of edges contained within them is thus at most
\begin{equation}\label{eq:sd-inner-edges}
\left\lceil \frac{n}{n'} \right\rceil \cdot \gamma_{d-1} \cdot n^{2^{d-1}/(2^d - 1)} \le 4 \cdot n^{2^{d-1}/(2^d -1)} \cdot \gamma_{d-1} \cdot n^{2^{d-1}/(2^d - 1)} = 4\cdot \gamma_{d-1} \cdot n^{2^d/(2^d -1)} %
\end{equation}
where the first inequality follows by applying~\eqref{eq:sd-num-subtrees}.

It remains to bound the number of extra edges introduced in the root of~$T$.
There are $\lceil n / n' \rceil$ subtrees, each containing only two leaves coloured by~$(1,i)$ for some~$i \in [2^{d-2}]$.
Thus the number of edges introduced is at most
\begin{equation}\label{eq:sd-outer-edges}
\binom{2 \cdot \lceil n / n' \rceil}{2} \le \frac{4 \cdot  \lceil n / n' \rceil^2}{2} \le 2 \cdot \left(4 \cdot n^{2^{d-1}/(2^d -1)}\right)^2 = 32 \cdot \left(n^{2^{d}/(2^d -1)}\right) %
\end{equation}
where we first applied the inequality $\binom{t}{2} \le \frac{t^2}{2}$, followed by an application of~\eqref{eq:sd-num-subtrees}.

Adding \eqref{eq:sd-inner-edges} and~\eqref{eq:sd-outer-edges} together, we obtain the desired bound on the number of edges in~$G$
\[|E(G)| \le  (4\cdot\gamma_{d-1} + 32)\cdot n^{2^{d}/(2^d -1)} = \gamma_d \cdot n^{2^{d}/(2^d -1)}.\qedhere\]
\end{proof}

\subsection{Distance to $\Pi$}\label{subsec:distance_to_pi}

We start this subsection with a general framework to prove lower bounds classes characterized by distance to $\Pi$ based on the lower bound for $\Pi$ itself.
Next we present two applications of the framework.
Then we complement it with similarly general framework to provide matching upper bounds.

\begin{theorem}\label{thm:transforming_to_distance_to_pi}
    Let $\Pi$ be a hereditary graph class. Suppose there is a constant $n_0$ such that every traceable $G\in \Pi$ on $n\geq n_0$ vertices contains at least $n\cdot \varphi(n)$ edges, where the function $n\mapsto n\cdot \varphi(n)$ is convex and $\varphi$ is unbounded and nondecreasing. For any $d\geq 0$, let $\Gamma_d$ be the class of graphs $G$ with $\operatorname{dist}_{\Pi}(G)\leq d$. Then every traceable $G\in \Gamma_d$ with $n\geq 2n_0\cdot \max\{d,1\}$ vertices contains at least $\frac{1}{2}n\cdot \varphi\left(\frac{n}{2(d+1)}\right)$ edges.
\end{theorem}
\begin{proof}
    If $d=0$, then $\Pi=\Gamma_d$ and if $G$ has at least $2n_0\geq n_0$ vertices, then it contains at least $n\cdot \varphi(n)\geq \frac{1}{2}n\cdot \varphi\left(\frac{n}{2}\right)$ edges, because $\varphi$ is nondecreasing.

    Assume $d\geq 1$. Let $G\in \Gamma_d$ be a graph with $n\geq 2n_0d$ vertices and let $S\subseteq V(G),|S|\leq d$ be a modulator to $\Pi$. Let $P$ be a Hamiltonian path of $G$. The modulator splits $P$ into $\ell \leq d + 1$ segments $P_1,\ldots,P_\ell$ with $n-|S|\geq n - d$ vertices in total. Denote $G_i=G[V(P_i)]$ and $p_i=|V(P_i)|$. Up to reordering, suppose that first $\ell'$ segments contain at least $n_0$ vertices. Note that, as the other segments contain at most $n_0-1$ vertices, we have $\ell'\geq 1$ as otherwise we would have 
    \[ n\leq (n_0-1)\ell + d\leq (n_0-1)(d+1)+d=n_0(d+1)-1\leq 2n_0d-1,\] contradicting $n\geq 2n_0d$. Hence $\ell - \ell' \leq d$. 
    By assumption, if $p_i\geq n _0$, then $|E(G_i)|\geq p_i\varphi(p_i)$. Applying Jensen's inequality to the convex function $x\mapsto x\cdot \varphi(x)$, we obtain the following estimate on the number of edges in $G$:
    \[
    |E(G)| \geq 
    \sum_{i=1}^{\ell'} |E(G_i)| \geq 
    \sum_{i=1}^{\ell'} p_i\varphi(p_i) \geq 
    \ell' \cdot \frac{\sum_{i=1}^{\ell'} p_i}{\ell'} \varphi\left(\frac{\sum_{i=1}^{\ell'} p_i}{\ell'}\right)\geq
    (n-dn_0)\cdot \varphi\left(\frac{n-dn_0}{\ell'}\right)\geq 
    \frac{1}{2}n\cdot \varphi\left(\frac{n}{2(d+1)}\right)
    \]
    The penultimate inequality follows from the fact that $\varphi$ is nondecreasing along with the bound $\sum_{i=1}^{\ell'}p_i\geq n - d-(\ell-\ell')\cdot (n_0-1)\geq n - dn_0$.    
    Finally, the last inequality follows from the fact that $\ell'\leq \ell \leq d + 1$ and that $n-dn_0\geq \frac{n}{2}$ because $n\geq 2dn_0$. 
    This completes the proof.
\end{proof}

\begin{corollary}\label{cor:lower_bound_cvdn}
    Let $G$ be a traceable graph on $n\geq 4 \cdot \max \{\cvdn(G),1\}$ vertices. Then $G$ contains at least $\frac{1}{16}\cdot \frac{n^2}{\cvdn(G)+1}$ edges.
\end{corollary}
\begin{proof}
    Any traceable cluster graph is a complete graph, hence for $n\geq 2$ vertices, it contains $\frac{n(n-1)}{2}\geq \frac{n^2}{4}$ edges. By \Cref{thm:transforming_to_distance_to_pi} for $\varphi(n)=\frac{n}{4}$, if $G$ has at least $4\cdot \max \{\cvdn(G),1\}$ vertices, then it contains at least $\frac{1}{16}\cdot \frac{n^2}{\cvdn(G)+1}$ edges.
\end{proof}

\begin{corollary}\label{cor:lower_bound_dtcog}
    Let $\Pi$ be the class of cographs.
    Let $G$ be a graph on $n\geq 2\cdot \max \{\operatorname{dist}_{\Pi}(G),1\}$ vertices. Then $G$ contains at least $\frac{1}{8}\cdot n\log_2 \frac{n}{2(\operatorname{dist}_\Pi(G)+1)}$ edges.
\end{corollary}
\begin{proof}
    By \Cref{thm:lower_bound_cograph} any traceable cograph on $n\geq 1$ vertices contains at least $\frac{1}{4}n\log_2n$ edges. By applying \Cref{thm:transforming_to_distance_to_pi} for the function $\varphi(n)=\frac{1}{4}\log _2n$ we obtain that if $G$ has at least $2\cdot \max \{\operatorname{dist}_{\Pi}(G),1\}$ vertices, then it contains at least $\frac{1}{8}n\log_2 \frac{n}{2(\operatorname{dist}_\Pi(G)+1)}$ edges.
\end{proof}

\begin{theorem}\label{thm:upper_bound_dist_to_pi_generic}
    Let $\Pi$ be a class closed under disjoint unions and assume that there exists infinitely many traceable graphs $G\in \Pi$ on $k$ vertices with at most $k\cdot \varphi(k)$ edges for some unbounded, nondecreasing function $\varphi\colon\mathbb{R}\to \mathbb{R}$.
    
    Then for each $d\geq 0$ there exists infinitely many graphs $G$ with $\operatorname{dist}_{\Pi}(G) \leq d$ on $n$ vertices and with at most $2 n\cdot \varphi(\frac{n}{d+1})$ edges.
\end{theorem}

\begin{proof}
    For infinitely many (and sufficiently large) $k\geq 1$ we construct a graph $G$ with $n=k(d+1)+d$ vertices with $\operatorname{dist}_{\Pi}(G)\leq d$. Construct $d+1$ identical copies $G_0,G_1,\ldots, G_d$ of traceable graphs from $\Pi$ on $k$ vertices and add $d$ vertices $w_1,\ldots,w_d$ and connect $w_i$ with the endpoints of the Hamiltonian paths in $G_{i-1}$ and $G_i$ by an edge. Clearly $G$ contains a Hamiltonian Path and $G\setminus \{w_1,\ldots, w_d\}\in \Pi$ since each $G_i$ is from $\Pi$ and $\Pi$ is closed under disjoint unions. The number of edges of $G$ is at most $(d+1)\cdot k\cdot \varphi(k)+2d\leq 2n\cdot \varphi(\frac{n}{d+1})$. We used the fact that $k\leq \frac{n}{d+1}$ and that $2d \leq n\cdot \varphi(\frac{n}{d+1})$ for $n$ large enough. 
\end{proof}

\begin{corollary}\label{cor:upper_bound_cvdn}
    For any given $d\geq 0$ and infinitely many $n$ there is a traceable graph $G$ on $n$ vertices with $\cvdn(G)\leq d$ and at most $2\cdot \frac{n^2}{d+1}$ edges.
\end{corollary}

\begin{corollary}\label{cor:upper_bound_dtcog}
    Let $\Pi$ be the class of cographs. For any given $d\geq 0$ and infinitely many $n$ there is a traceable graph $G$ on $n$ vertices with $\operatorname{dist}_{\Pi}(G)\leq d$ and at most $2n\log \frac{n}{d + 1}$ edges.
\end{corollary}

\section{Conclusion}
In this paper, we bound the minimum number of edges of traceable graphs that belong to several structured classes of  graphs. We show that the bound is quadratic for the class of bounded neighborhood diversity and bounded cluster vertex deletion number, $n\log n$ for bounded modular-width and slightly superlinear for the class of bounded shrub-depth.
Together with the trivial linear bounds for distance to path (and therefore, e.g., treewidth) and for domination number and with the constant bound on the order of traceable graphs of bounded treedepth, gives more or less complete picture for vast majority of popular graph width parameters.

We note that our bounds can easily be adapted to Hamiltonian graphs (i.e., graphs containing a Hamiltonian cycle) -- each such graph is also traceable. We did not find a  class where considering a Hamiltonian cycle instead of Hamiltonian path would make a difference. For example, the path on $n$ vertices is a chordal graph. However, requiring a Hamiltonian cycle instead does not improve the lower bound significantly.  The triangulated $2\times n$ grid has $O(n)$ edges, is chordal, and has a Hamiltonian cycle (in fact, it is even a proper interval graph).

Our bounds do not hold if we instead of a Hamiltonian path require the graph to be Eulerian.
Intuitively, in the case of Eulerian graphs, we lose the property that every vertex has a unique successor on the path, so we are able to obtain smaller upper bounds. In particular, for even $n$, the complete bipartite graph $K_{n-2,2}$ is Eulerian and has linearly many edges. The class of stars shows that also just connectivity is not enough to obtain non-trivial bounds.

One possible direction for further research is to consider traceable graphs that are $H$-free for some fixed graph $H$. Note that if $H$ contains a vertex of degree at least 3 or a cycle, the class of $H$-free graphs contains the path on $n$ vertices for infinitely many $n$, so the bound becomes $\Theta(n)$. The remaining case is when $H$ is a disjoint union of paths. For $H=P_3$ we obtain precisely the class of clusters for which the bound is $\Theta(n^2)$. For $H=P_4$, we obtain precisely the class of cographs, where the bound is $\Theta(n\log n)$ but the case $H=P_k$ for $k\geq 5$ remains open. On the other hand, $H=dP_1$ corresponds to graphs with no independent set on $d$ vertices, and $H=dP_2$ corresponds to graphs with no induced matching of size $d$ (i.e., $\mim < d$). The case of $H=dP_k$ for $k\geq 3$ remains open. Note that if a graph is $dP_k$-free, then it is also $P_{d(k+1)}$-free (but not vice versa).

\bibliography{references}

\end{document}